\documentclass{article}
\usepackage{latexsym}
\usepackage{amsmath}
\usepackage{amsthm}
\usepackage{amscd}
\usepackage{amsfonts}
\usepackage{amssymb}
\usepackage{dsfont}
\usepackage{bm} 
\usepackage{bbm} 
\usepackage{float} 

\theoremstyle{plain}   
\theoremstyle{plain}   \newtheorem{Lem}{Lemma}
\theoremstyle{plain} 	
\theoremstyle{plain} 	\newtheorem{The}{Theorem}
\theoremstyle{plain} 	\newtheorem{Prop}{Proposition}
\theoremstyle{plain} 	
\theoremstyle{plain}	\newtheorem*{Rem}{Remark}
\theoremstyle{plain}	\newtheorem*{Rems}{Remarks}
\theoremstyle{plain}	\newtheorem{Def}{Definition} 
\theoremstyle{plain}	
\theoremstyle{plain}   \newtheorem{Cla}{Claim}
\theoremstyle{plain} 
\theoremstyle{plain} 

\def\ceil#1{\left\lceil#1\right\rceil} 
\def\floor#1{\left\lfloor#1\right\rfloor} 

\def\iff{\Longleftrightarrow}
\def\indicator{\mathds{1}}

\def\Var{\mbox{Var}}

\def\iB{\textbf{i}}

\def\E{\mathbb{E}}

\def\N{\mathbb{N}}

\def\P{\mathbb{P}}

\def\R{\mathbb{R}}

\def\Z{\mathbb{Z}}

\def\SM{\mathcal{S}}
\def\TM{\mathcal{T}}

\def\Ct{\widetilde{C}}

\def\Int{I_t^{(n)}} 
\def\In{I^{(n)}}
\def\Xn{X^{(n)}}
\def\Xnt{X_t^{(n)}}

\begin{document}

\title{Inversions and Longest Increasing Subsequence for $k$-Card-Minimum Random Permutations}

\author{Nicholas F. Travers \thanks{Department of Mathematics, Technion--Israel Institute of Technology.  
E-mail - \texttt{travers@tx.technion.ac.il}. }}

\date{} 
\maketitle

\vspace{-4 mm}
\begin{abstract}

A random $n$-permutation may be generated by sequentially removing random cards $C_1,...,C_n$ from 
an $n$-card deck $D = \{1,...,n\}$. The permutation $\sigma$ is simply the sequence of cards in the order they are removed.
This permutation is itself uniformly random, as long as each random card $C_t$ is drawn uniformly from the remaining set at 
time $t$. We consider, here, a variant of this simple procedure in which one is given a choice between $k$ random cards 
from the remaining set at each step, and selects the lowest numbered of these for removal. This induces 
a bias towards selecting lower numbered of the remaining cards at each step, and therefore leads to a final 
permutation which is more ``ordered'' than in the uniform case (i.e. closer to the identity permutation id $=(1,2,3,...,n)$).

We quantify this effect in terms of two natural measures of order: The number of inversions $I$ and the length
of the longest increasing subsequence $L$. For inversions, we establish a weak law of large numbers and central limit
theorem, both for fixed and growing $k$. For the longest increasing subsequence, we establish the rate of scaling,
in general, and existence of a weak law in the case of growing $k$. We also show that the minimum strategy, of 
selecting the minimum of the $k$ given choices at each step, is optimal for minimizing the number of inversions
in the space of all online $k$-card selection rules.  
\end{abstract}

\section{Introduction}
\label{sec:Introduction}

A random $n$-permutation may be generated with a deck of $n$ cards $D = \{1,...,n\}$ as follows. Draw a random card $C_1$ 
from the deck and remove it, then draw and remove another random card $C_2$ from the remaining cards, and so forth until all 
$n$ cards have been removed. The permutation is $\sigma = (C_1,...,C_n)$, where $C_t$ is the card removed at time $t$. 
This permutation is itself uniformly random, as long as each random card $C_t$ is drawn uniformly from the remaining
cards in the deck at time $t$. 

If, however, one is given a choice between $k \geq 2$ (uniformly) random cards to remove at each step
one can bias the resulting permutation by an appropriate selection rule to achieve a particular objective. For example, 
one can seek to maximize the number of fixed points, number of cycles, length of the the longest cycle, ... etc. Our aim 
here is to create a permutation which is as ``ordered'' as possible. That is, closest to the identity permutation 
id $ = (1,2,3,...,n)$. For this, we choose the natural strategy of selecting the lowest numbered, or minimum, 
of the $k$ random card choices at each step. 

We refer to the resulting procedure for constructing our random permutation as the $k$-card-minimum procedure. 
Formally, it is defined below. 

\begin{Def}
\label{def:kCardMinProcedure}
For $k,n \in \N$ the \emph{$k$-card-minimum ($k$CM) procedure} is the following random 
algorithm for generating a permutation $\sigma$ of the integers $1,...,n$.
\begin{align*}
& \bullet ~ D_1 = \{1,...,n\} \\
& \bullet ~ \mbox{For } t = 1,...,n: \\
& ~~~~~~~ C_{t,1},...,C_{t,k} \mbox{ are i.i.d uniform samples from } D_t \\
& ~~~~~~~ C_t = \min \{C_{t,1},...,C_{t,k} \} \\
& ~~~~~~~ D_{t+1} = D_t/\{C_t\} \\
& \bullet ~ \sigma = (C_1,...,C_n)
\end{align*}
\end{Def}

Here, $D_t$ represents the set of cards remaining in the deck at time $t$, just before the $t$-th card is selected. $C_{t,1}, ... , C_{t,k}$ 
are the $k$ random card choices from the remaining set $D_t$, and the minimum of these, $C_t$, is selected for removal.  
The final permutation $\sigma = (C_1,...,C_n)$ is simply the sequence of cards in the order they are removed. 

With $k=1$, of course, the $k$CM procedure reduces to the original procedure in which a single random card is drawn at each step, 
and the final permutation $\sigma$ is uniform. However, for any $k \geq 2$ one expects the selection rule to create a more ordered 
permutation. We allow the case $k=1$ in the definition only because it facilitates easy comparison to the uniform case from our theorems, 
and does not add any increased difficulty in the proofs. 

\subsection{Measures of Order} 
\label{subsec:MeasuresOfOrder}

The extent to which a permutation is ``ordered'' is not, a priori, a well-defined mathematical concept, but 
we will consider two natural measures of order for our analysis: The number of inversions $I$ and 
the length of the longest increasing subsequence $L$.

\begin{Def}
For an $n$-permutation $\sigma = (\sigma(1),...,\sigma(n)):$
\begin{align*}
I(\sigma) & = |\{i<j : \sigma(i) > \sigma(j)\}| \mbox{ and }\\
L(\sigma) & = \max\{ \ell : \exists~ 1 \leq i_1 < ... < i_{\ell} \leq n \mbox{ with } \sigma(i_1) < ... < \sigma(i_{\ell}) \}.
\end{align*}
\end{Def}

Intuitively, of course, a more ordered permutation should have fewer inversions and a longer longest 
increasing subsequence, and, in fact, this intuition can be justified concretely in the following sense. 
\begin{align*}
I(\sigma) = d_{AT}(\sigma,\mbox{id}) ~\mbox{ and } L(\sigma) = n - d_R(\sigma,\mbox{id})
\end{align*}
where $d_{AT}$ and $d_{R}$ are the standard permutation metrics defined by
\begin{align*}
& d_{AT}(\sigma, \sigma') = \mbox{min $\#$ adjacent transpositions required to transform $\sigma$ into $\sigma'$}, \\
& d_{R}(\sigma, \sigma') = \mbox{min $\#$ reinsertions required to transform $\sigma$ into $\sigma'$}.
\end{align*} 
Here, as usual, two permutations $\sigma$ and $\sigma'$ are said to differ by a single adjacent 
transposition if they are of the form 
\begin{align*}
\sigma = (i_1,...,i_n) ~,~ \sigma' = (i_1,...,i_{m-1}, \iB_{\textbf{m+1}}, \iB_{\textbf{m}}, i_{m+2},...,i_n)
\end{align*}
and to differ by a single reinsertion if they are of the form
\begin{align*}
\sigma = (i_1,...,i_n) ~,~ \sigma' = (i_1,...,i_{m-1},  \iB_{\textbf{m+j}}, i_m,...,i_{m+j-1}, i_{m+j+1},...,i_n)
\end{align*} 
or the form
\begin{align*}
\sigma = (i_1,...,i_n) ~,~ \sigma' = (i_1,...,i_{m-1}, i_{m+1},...,i_{m+j},\iB_{\textbf{m}}, i_{m+j+1},...,i_n). 
\end{align*}

\subsection{Summary of Results and Comparison to Uniform Case} 
\label{subsec:SummaryOfResultsComparisonUniform}

For a uniformly random permutation, $I \sim n^2/4$ and converges to a standard normal distribution when appropriately 
centered and rescaled \cite{Feller1968}, whereas $L \sim 2 \sqrt{n}$~ \cite{Vershik1977, Logan1977, Aldous1995} and converges to a 
Tracy-Widom distribution when appropriately centered and rescaled \cite{Baik1999}. For fixed $k$, we find that $I$ and $L$ still obey the 
same $n^2$ and $\sqrt{n}$ scalings, but decreased and increased, respectively, by constant factors. However, if $k = k_n \rightarrow \infty$ 
then the scaling rates are altered. In particular, if $k_n \rightarrow \infty$ with $k_n = o(n)$, then $I$ scales as $n^2/k_n$ and $L$ scales as 
$\sqrt{k_n n}$. More precise statements, including weak laws and central limit theorems, will be given below in Section \ref{sec:StatementOfResults}. 

\subsection{Motivation and Related Work} 
\label{subsec:MotivationAndRelatedWork} 

If $n$ balls are placed into $n$ bins independently and uniformly at random, then the number of balls in the fullest bin or \emph{maximum load}
is roughly $\log(n)/\log \log(n)$ with high probability. If, however, the balls are placed sequentially, and at each step one is allowed to choose from 
among $k$ independent randomly selected bins, then the maximum load can be reduced dramatically to $\log \log(n) / \log(k)$, by always choosing 
to place the ball in the least full bin of the given choices \cite{Azar1999}. This is one of the first, and most remarkable, examples of the power of 
choice in stochastic models. 

Another important example is the Achlioptas model, which is a modification of the standard Erd\"{o}s-R\'{e}nyi random graph process 
$(G(n,m))_m$, in which one is allowed to select from among $k$ independently chosen random edges to add to the $n$-vertex graph at 
each step, rather than simply adding a given random edge. Using appropriate selection rules with $k = 2$ in this model, one can accelerate or 
delay the onset of the giant component from the Erd\"{o}s-R\'{e}nyi critical point of $m = n/2$ edges by a constant factor: to as early as $0.385 n$ 
or as late as $0.829 n$ \cite{Bohman2006, Spencer2007}. Using other selection rules one can also substantially delay (with fixed $k$) or accelerate
(with growing $k$) the first appearance time $m_{H}$ of a fixed subgraph $H$ \cite{Krivelevich2009, Krivelevich2012}. 

Our random permutation model is, of course, mathematically quite different than either the balls and bins selection model 
or the Achlioptas random graph model, but the questions we are interested in are very similar in spirit. We begin with 
a well studied base model, the uniform permutation, which can be generated by a sequential procedure, removing cards one at 
a time. Then, we add choice to the procedure with the goal of modifying some statistical property of the resulting random object. 
In particular, we wish to make the final permutation more ordered, so we select at each step the lowest numbered card from 
among the $k$ given choices. 

This is a very simple strategy, essentially a greedy algorithm. However, as we will show below (Proposition 
\ref{eq:MinStrategyOptimalForI}) it is, in fact, optimal for minimizing the number of inversions $I$, just as the simple 
greedy strategy of selecting the least full bin from among the $k$ choices is optimal for reducing the maximum load 
in the balls and bins model \cite{Azar1999}. For maximizing $L$ our greedy selection rule is not optimal, but it still 
substantially increases $L$ for large fixed $k$ or growing $k$. 

Another motivation for the study of $k$CM random permutations comes from the Mallows random permutation model 
\cite{Mallows1957}. We expect the $k$CM model to have a similar band structure to the Mallows model \cite{Bhatnagar2013}, 
and our theorems, along with the previous work on Mallows permutations in \cite{Bhatnagar2013, Rabinovitch2012}, 
show that with an appropriate choice of parameters both $I$ and $L$ scale at the same rate in the two models. That is, if one 
chooses parameters to ensure roughly the same number of inversions, then one also gets roughly the same length of 
longest increasing subsequence. 
 
\section{Statement of Results}
\label{sec:StatementOfResults}

In this section we state formally our results for the statistics of inversions and the longest increasing subsequence in $k$CM
random permutations. These are divided into four subsections: inversions results for fixed $k$, inversion results for growing $k$,
scaling results for $L$ (both for fixed and growing $k$), and optimality results for the minimum strategy. Proofs will be given later 
in Sections \ref{sec:AnalysisOfInversions}, \ref{sec:AnalysisOfLIS}, and \ref{sec:AnalysisOfOptimality_kCMProcedure}.

Throughout we use the following notation:
\begin{itemize}
\item $[n] = \{1,...,n\}$.
\item $\stackrel{p.}{\longrightarrow}$ and $\stackrel{d.}{\longrightarrow}$ denote, respectively, convergence in probability 
and convergence in distribution. 
\item $\sigma = (C_1,...,C_n)$ is a random permutation generated according to the $k$CM procedure of Definition 
\ref{def:kCardMinProcedure}. 
\item $I$ is the number of inversions in $\sigma$, and $L$ is the length of the longest increasing subsequence in $\sigma$.
\item $\P_{n,k}$ is the probability measure when the $k$CM procedure is run on a deck of $n$ cards with given $k$.  
\item $\E_{n,k}(X)$ and $\Var_{n,k}(X)$ denote, respectively, the expectation and variance of a random variable $X$ 
under the measure $\P_{n,k}$. 
\end{itemize} 

\subsection{Inversion Results for Fixed $k$}
\label{subsec:InversionResultsFixedk}

For $k \in \N$, let
\begin{align*}
a_k = \frac{1}{2(k+1)} ~~\mbox{ and }~~ b_k = \frac{k}{3(k+1)^2(k+2)} ~.
\end{align*} 
Then, we have the following asymptotics for $\E_{n,k}(I)$ and $\Var_{n,k}(I)$,
as $n$ goes to infinity with fixed $k$. 

\begin{Prop}
\label{prop:ExpAndVarI_Fixedk}
For any fixed $k \in \N$,
\begin{align}
\label{eq:ExpAndVarI_Fixedk}
\E_{n,k}(I) = a_k n^2 + ~ O(n) ~~\mbox{ and }~~ \Var_{n,k}(I) = b_k n^3 + O(n^2).
\end{align}
\end{Prop}

Moreover, a weak law of large numbers and central limit theorem both hold. 

\begin{The}[Weak Law of Large Numbers] 
\label{thm:WeakLawFixedk}
For any fixed $k \in \N$,
\begin{align}
\label{eq:WeakLawFixedk}
\frac{I}{n^2} \stackrel{p.}{\longrightarrow} a_k, \mbox{ as } n \rightarrow \infty. 
\end{align}
\end{The}

\begin{The}[Central Limit Theorem] 
\label{thm:CentralLimitTheoremFixedk}
For any fixed $k \in \N$, 
\begin{align}
\label{eq:CentralLimitTheoremFixedk_1}
\frac{I - \E_{n,k}(I)}{\sqrt{\Var_{n,k}(I)}} \stackrel{d.}{\longrightarrow} Z ~,~ \mbox{ as } n \rightarrow \infty
\end{align}
where $Z$ is a standard normal random variable. Equivalently, 
\begin{align}
\label{eq:CentralLimitTheoremFixedk_2}
\frac{I - a_k \cdot n^2}{\sqrt{ b_k } \cdot n^{3/2}} \stackrel{d.}{\longrightarrow} Z ~,~ \mbox {as } n \rightarrow \infty.
\end{align}
\end{The}

\subsection{Inversion Results for Growing $k$}
\label{subsec:InversionResultsGrowingk}

Throughout this section we assume that $(k_n)_{n = 1}^{\infty}$ is a nondecreasing sequence of positive integers such that
$k_n \rightarrow \infty$ with $k_n = o(n)$. Our first proposition gives asymptotic estimates for the expectation and variance 
of the number of inversions $I$, analogous to Proposition \ref{prop:ExpAndVarI_Fixedk}. 

\begin{Prop}
\label{prop:ExpAndVarI_Growingk}
\begin{align}
\label{eq:ExpAndVarI_Growingk}
\E_{n,k_n}(I) = \frac{1}{2} \cdot \frac{n^2}{k_n} ~+~ o\left( \frac{n^2}{k_n}\right) ~~\mbox{ and }~~
\Var_{n,k_n}(I) = \frac{1}{3} \cdot \frac{n^3}{k_n^2} ~+~ o\left( \frac{n^3}{k_n^2}\right).
\end{align}
\end{Prop}

Using this proposition, along with some intermediate estimates used in its proof, we also obtain a 
weak law of large numbers and central limit theorem for the number of inversions $I$, analogous 
to Theorems \ref{thm:WeakLawFixedk} and \ref{thm:CentralLimitTheoremFixedk}. 

\begin{The}[Weak Law of Large Numbers] 
\label{thm:WeakLawGrowingk}
\begin{align}
\label{eq:WeakLawGrowingk}
I \cdot \frac{k_n}{n^2} \stackrel{p.}{\longrightarrow} \frac{1}{2}, \mbox{ as } n \rightarrow \infty. 
\end{align}
\end{The}

\begin{The}[Central Limit Theorem] 
\label{thm:CentralLimitTheoremGrowingk}
\begin{align}
\label{eq:CentralLimitTheoremGrowingk_1}
\frac{I - \E_{n,k_n}(I)}{\sqrt{\Var_{n,k_n}(I)}} \stackrel{d.}{\longrightarrow} Z ~,~ \mbox{ as } n \rightarrow \infty
\end{align}
where $Z$ is a standard normal random variable. 
\end{The}

\subsection{Scaling of $L$}
\label{subsec:ScalingOfL} 

The following theorem is our primary result for the longest increasing subsequence in
$k$CM random permutations. It establishes the scaling rate of $L$ as $n \rightarrow \infty$ 
up to a universal constant factor, both for fixed and growing $k$. 

\begin{The}
\label{thm:ScalingOfL}
If $(k_n)_{n=1}^{\infty}$ is any sequence of positive integers satisfying $k_n = o(n)$, then
\begin{align}
\label{eq:ExScalingOfL}
1/2 \leq \liminf_{n \to \infty}~ \frac{\E_{n,k_n}(L)}{\sqrt{k_n n}} \leq \limsup_{n \to \infty}~ \frac{\E_{n,k_n}(L)}{\sqrt{k_n n}} \leq 4e.
\end{align}
Moreover, for any $\epsilon > 0$,
\begin{align}
\label{eq:ScalingOfL}
\lim_{n \to \infty} \P_{n,k_n} \left( 1/2 - \epsilon \leq L/\sqrt{k_n n} \leq 4e + \epsilon \right) = 1.
\end{align}
\end{The}

In the case $k_n \rightarrow \infty$, we also obtain existence of a weak law of large numbers, 
though we do not know the exact constant for the weak law. 

\begin{The}[Weak Law of Large Numbers]
\label{thm:WeakLawL}
If $(k_n)_{n=1}^{\infty}$ is any sequence of positive integers such that $k_n \rightarrow \infty$ with $k_n = o(n)$, then
\begin{align*}
\frac{L}{\E_{n,k_n}(L)} \stackrel{p.}{\longrightarrow} 1 ~,~ \mbox { as } n \rightarrow \infty.
\end{align*}
\end{The}

A central piece of the proof of Theorem \ref{thm:WeakLawL} is the following variance estimate for 
$L$, which, interestingly, does not depend of $k$. 

\begin{Prop}
\label{prop:VarL}
For any $k,n \in \N$,
\begin{align}
\label{eq:VarL}
\Var_{n,k}(L) \leq n/4.
\end{align}
\end{Prop}

\begin{Rems} ~
\begin{enumerate}
\item The constants $1/2$ and $4e$ in Theorem \ref{thm:ScalingOfL} can be improved a bit by a somewhat more careful 
analysis than we give here. However, we do not believe they can be made to match without substantially different methods. 
\item Our theorem shows that $L$ is increased by roughly a factor of $\sqrt{k_n}$ from the uniform $2\sqrt{n}$ scaling,
up to moderate corrections. However, for $k_n = k$ fixed and small these moderate corrections are of the same order as 
the $\sqrt{k}$ increase. Thus, the theorem is most informative only for fixed large $k$ or growing $k$. 
\item It is natural to consider how fast the sequence $(k_n)$ must grow to increase the scaling rate of $L$ from the order $n^{1/2}$ 
uniform scaling to a larger power law $n^{\alpha}$, for some $\alpha > 1/2$. According to our theorem, one must also take 
$k_n$ as a power law, $k_n \approx n^{\beta}$ where $\alpha = 1/2 + \beta/2$. Since any $\beta \in (0,1)$ is possible, while 
still maintaining $k_n = o(n)$, any $\alpha \in (1/2,1)$ is also possible. 
\end{enumerate}
\end{Rems}

\subsection{Optimality Results for the Minimum Strategy} 
\label{subsec:OptimalityResults}

In the $k$CM procedure we are given $k$ independent random card choices $C_{t,1},...,C_{t,k}$ from the remaining set $D_t$ 
at each step, and we select the lowest numbered of these. This selection rule was chosen in order to create a more ordered final permutation, 
and it is a simple and natural rule for doing so. However, it is reasonable to ask if some other selection rule may be better for this purpose. 
More generally, given any real-valued statistic $X = X(\sigma)$ to maximize or minimize (for us $X = I$ or $L$), one may ask if a given 
selection rule or \emph{strategy} is optimal for maximizing or minimizing this statistic. 

The $k$CM strategy $\SM_{\min}$, in which the minimum of the $k$ card choices is always selected, is quite simple. 
It does not depend explicitly on $n$ or $k$, or on any of the previously removed cards $C_1, ..., C_{t-1}$. In principal, though, 
it is reasonable to allow a strategy to depend explicitly on both $n$ and $k$, as well as the cards $C_1,...,C_{t-1}$ removed 
before time $t$, as these will be known to an individual making the card selections. We, thus, define a 
\emph{$(k,n)$ choice strategy} as follows. 

\begin{Def}
For $k,n \in \N$, a \emph{$(k,n)$ choice strategy} $\SM$ is an $n$-tuple of choice functions $\SM = (f_1,...,f_n)$ 
where for each $t \in [n]:$
\begin{itemize}
\item The domain of $f_t$ is the set of allowable input pairs $((c_1,...,c_{t-1}),(c_{t,1},...,c_{t,k}))$ such that each $c_{\tau}, c_{t,i} \in [n]$, 
$c_{\tau} \not= c_{\tau'}$ for all $\tau \not= \tau'$, and $c_{t,i} \not= c_{\tau}$ for each $1 \leq i \leq k$ and $1 \leq \tau \leq t-1$.
\item $f_t [ ((c_1,...,c_{t-1}),(c_{t,1},...,c_{t,k})) ] \in \{c_{t,1},...,c_{t,k}\}$ for each possible input \\$((c_1,...,c_{t-1}),(c_{t,1},...,c_{t,k}))$. 
\end{itemize}
On the event $\{C_{\tau} = c_{\tau}, 1 \leq \tau \leq t-1 ~\mbox{ and }~ C_{t,i} = c_{t,i}, 1 \leq i \leq k\}$ the strategy 
$\SM$ selects $C_t \in \{c_{t,1},...,c_{t,k}\}$ by the rule $C_t = f_t [ ((c_1,...,c_{t-1}),(c_{t,1},...,c_{t,k})) ]$.
\end{Def}

We say a (k,n) choice strategy $\SM$ is \emph{stochastically optimal for maximizing} a real-valued statistic $X = X(\sigma)$ if
for every other $(k,n)$ choice strategy $\hat{\SM}$ we have
\begin{align*}
\P_{n,k}^{\SM}(X \geq x) \geq \P_{n,k}^{\hat{\SM}}(X \geq x) ~,~ \mbox{ for all } x \in \R.
\end{align*}
Similarly, we say a (k,n) choice strategy $\SM$ is \emph{stochastically optimal for minimizing} a real-valued statistic $X = X(\sigma)$ if
for every other $(k,n)$ choice strategy $\hat{\SM}$ we have
\begin{align*}
\P_{n,k}^{\SM}(X \leq x) \geq \P_{n,k}^{\hat{\SM}}(X \leq x) ~,~ \mbox{ for all } x \in \R.
\end{align*}

The following proposition shows that the simple strategy $\SM_{\min}$ is actually the best strategy for minimizing inversions. 

\begin{Prop}
\label{eq:MinStrategyOptimalForI}
For each $k,n \in \N$, the $k$-card-minimum strategy $\SM_{\min}$ is stochastically optimal for minimizing $I$.
\end{Prop}

The situation for the longest increasing subsequence is more complicated, though, and $\SM_{min}$ is no
longer optimal. In fact, no optimal strategy exists for any $k \geq 2$ and $n \geq 5$. 

\begin{Prop} 
\label{eq:MinStrategyNotOptimalForL}
For any $k \geq 2$ and $n \geq 5$ there is no $(k,n)$ stochastically optimal strategy for maximizing $L$. 
However, for each $k \geq 2$ and $n \geq 4$ there exists a strategy $\SM_{copy}$, which is strictly 
better than the $k$-card-minimum strategy for maximizing $L$. That is, 
\begin{align*}
\P_{n,k}^{\SM_{copy}}(L \geq x) \geq \P_{n,k}^{\SM_{\min}}(L \geq x)
\end{align*}
for all $x \in \R$, with strict inequality for some values of $x$. 
\end{Prop}

Nevertheless, as shown by Theorem \ref{thm:ScalingOfL}, the minimum strategy $\SM_{\min}$ still increases $L$ 
substantially compared to the uniform case for large fixed $k$ or growing $k$. Moreover, although the minimum strategy is 
strictly dominated by $\SM_{copy}$, this dominance is very weak. As the name suggests, $\SM_{copy}$ copies $\SM_{\min}$ 
almost all of the time, and only behaves differently in very specific instances, for which it can increase $L$ by 1. Thus, $\SM_{\min}$
may still be ``essentially optimal'' in the sense of the scaling rate given by Theorem \ref{thm:ScalingOfL}. A natural question, for which 
we do not yet have an answer, is whether the minimum strategy is, indeed, optimal in terms of this scaling rate. More precisely: \\

\noindent
\textbf{Question} - Does there exist some absolute constant $B > 0$  such that for any sequence $(k_n)_{n=1}^{\infty}$ with $k_n = o(n)$ 
and any sequence of $(k_n,n)$ choice strategies $(\SM_n)_{n=1}^{\infty}$, 
\begin{align*}
\lim_{n \to \infty} \P_{n,k_n}^{\SM_n} \left(L \leq B \sqrt{k_n n}\right) = 1 ~?
\end{align*}

\section{Analysis of Inversions}
\label{sec:AnalysisOfInversions}

In this section we analyze statistics of the number of inversions $I$ in a $k$CM random permutation, 
proving the results of Sections \ref{subsec:InversionResultsFixedk} and \ref{subsec:InversionResultsGrowingk}. 
We treat separately the case of fixed $k$ in Section \ref{subsec:FixedkCase} and the case of growing $k$ in Section 
\ref{subsec:GrowingkCase}. First, however, we begin with some general set up in Sections \ref{subsec:Preliminaries} 
and \ref{subsec:EstimateEandVarIt} that will be used in both cases. 

\subsection{Preliminaries}
\label{subsec:Preliminaries}

\subsubsection{Relative Positions}
\label{subsubsec:Relative Positions}

The sequence $(C_t)_{t=1}^n$ is highly dependent. Partly this is because of simple exclusion; if $C_t = i$ then 
$C_{\tau} \not= i$, for all $\tau \not= t$. However, unlike in the uniform case, there is strong dependence beyond this 
as well since the kCM procedure has a bias towards selecting lower numbered of the remaining cards at each step. 
Theoretically at least, at any time $t > \max\{i-1,n-i\}$ card $i$ could be the lowest or highest remaining card in the deck,
or anything in between. So, the probability that $C_t = i$ (assuming $i$ is still left at time $t$) depends heavily on 
which cards were removed at earlier times $\tau < t$. The main idea for analyzing inversions is to consider 
the relative card positions, which are independent, and, thus, circumvent this difficulty. 

For $t  = 1,...,n$ we define $\Ct_t$ to be the \emph{relative position} of card $C_t$ in the remaining 
set of cards $D_t$ from which it is selected: 
\begin{align*}
\Ct_t = j \iff C_t \mbox{ is the $j$-th lowest numbered card in $D_t$}.
\end{align*}
Since the selection rule for the $k$CM procedure depends only on the relative values of the $k$ card 
choices $C_{t,1},...,C_{t,k}$, and not on the actual numbers of these cards, the relative position $\Ct_t$ 
of the $t$-th card selected is independent of all cards removed up to time $t$, $C_1,...,C_{t-1}$. Hence, also, 
independent of the relative positions of the previously removed cards, $\Ct_1, ... , \Ct_{t-1}$. Since this holds 
for each $t = 1,...,n$, it  follows that the sequence of relative card positions $\Ct_1,...,\Ct_n$ is independent, 
as claimed above. 

The relation to inversions is as follows. If we define
\begin{align*}
I_t = |\{t+1\leq \tau \leq n : C_t > C_{\tau} \}|
\end{align*}
to be the number of cards selected at later times $\tau > t$, which are inverted with card $C_t$, then
\begin{align}
\label{eq:ItCttRelation}
I_t = \Ct_t - 1.
\end{align}
To see this, note that if the $j$-th lowest numbered of the remaining cards is selected at time $t$ (i.e. $\Ct_t = j$),
then there are exactly $j-1$ lower numbered cards remaining in the deck at time $t+1$, which eventually must be 
removed at times $\tau \geq t+1$. Thus, there will be exactly $j-1$ cards removed at times $\tau \geq t+1$, 
which are inverted with card $C_t$ (i.e. $I_t = j-1$). 

From (\ref{eq:ItCttRelation}) and independence of the relative positions $\Ct_1,...,\Ct_n$
it follows that $I_1,...,I_n$ are independent as well. These facts are summarized below
in the following proposition, which also characterizes the distribution of the random variables
$\Ct_t$ and $I_t$. 

\begin{Prop}
\label{prop:IndependenceAndDistributionItAndCtt}
For any $k,n \in \N$, the random variables $\Ct_1,...,\Ct_n$ are independent and the random variables $I_1,...,I_n$ are independent. 
Moreover, for each $t = 1,...,n$ we have
\begin{align}
\label{eq:DistributionCtt}
\P_{n,k}(\Ct_t > j) = \left( \frac{n-t+1-j}{n-t+1} \right)^k ~,~j = 0,1,...,n-t+1 
\end{align}
and
\begin{align}
\label{eq:DistributionIt}
\P_{n,k}(I_t > j) = \left( \frac{n-t-j}{n-t+1} \right)^k ~,~j = 0,1,...,n-t. 
\end{align}
\end{Prop}

\begin{proof}
Independence of the $\Ct_t$'s and independence of $I_t$'s was established above, and 
(\ref{eq:DistributionIt}) follows from (\ref{eq:ItCttRelation}) and (\ref{eq:DistributionCtt}).
Thus, it remains only to prove (\ref{eq:DistributionCtt}). 

To see (\ref{eq:DistributionCtt}), note that at time $t$ there are exactly $n-t+1$ cards left in the deck 
to pick from (i.e. $|D_t| = n - t + 1$), and $\Ct_t > j$ if and only if each of the $k$ independent random choices
$C_{t,1},...,C_{t,k}$ is greater than $D_{t,j} \equiv j$-th lowest card in $D_t$. Thus,  
\begin{align*}
\P_{n,k}(\Ct_t > j) = \prod_{i=1}^k \P_{n,k}(C_{t,i} > D_{t,j}) = \prod_{i=1}^k \frac{(n-t+1) - j}{n-t+1}.
\end{align*}
\end{proof}

Now, of course, the total number of inversions $I$ is simply
\begin{align}
\label{eq:IequalssumIts}
I = \sum_{t=1}^{n-1} I_t.
\end{align}
So, by Proposition \ref{prop:IndependenceAndDistributionItAndCtt}, we have $I$ expressed
as a sum of independent random variables $I_t$, $t=1,...,n-1$, with explicit distribution 
(\ref{eq:DistributionIt}). The analysis of $I$ (both for fixed and growing $k$) is based
upon this decomposition. 

\subsubsection{The Lindberg-Feller Central Limit Theorem}
\label{subsubsec:LindbergFellerCLT}

The $I_t$'s are independent, but not identically distributed. Our proof of the central limit
theorem for the number of inversion $I$ (both in the case of fixed and growing $k$) will
use the following general central limit theorem for sums of independent random variables. 
See, e.g., \cite{Durrett2010}. 

\begin{The}[Lindberg-Feller Central Limit Theorem]
\label{thm:LindbergFellerCLT}
Let $\Xnt$, $n \in \N$ and $1 \leq t \leq n$, be independent mean zero random variables such that:
\begin{itemize}
\item[(i)] For each $n \in \N$, $\sum_{t=1}^n \Var \left(\Xnt \right) = 1$.
\item[(ii)] For each $\epsilon > 0$, $\lim_{n \rightarrow \infty} \left\{ \sum_{t=1}^n \E \left( \left|\Xnt\right|^2 ; \left|\Xnt\right| > \epsilon \right) \right\} = 0$.  
\end{itemize}
Then $\sum_{t=1}^n \Xnt \equiv \Xn \stackrel{d.}{\longrightarrow} Z$, as $n \rightarrow \infty$, where $Z$ is a
standard normal random variable.
\end{The} 

\subsubsection{Some Basic Estimates for Sums}
\label{subsubsec:BasicEstimates}

For calculation of the expectation and variance of the random variables $I_t$ 
we will need the following basic estimates for sums.

\begin{Lem}
\label{lem:BasicSumEstimates}
For positive integers $m,k$ 
\begin{align}
\label{eq:sum_tk_bounds}
\frac{m^{k+1}}{k+1} ~\leq~ \sum_{\tau = 1}^m \tau^k ~\leq~ \frac{(m+1)^{k+1}}{k+1}
\end{align}
and 
\begin{align}
\label{eq:sum_mmttk_bounds}
\frac{(m+1)^{k+2}}{(k+1)(k+2)} - \frac{4 (m+1)^{k+1}}{k+1} 
~\leq~ \sum_{\tau = 1}^m (m - \tau) \tau^k 
~\leq~ \frac{m^{k+2}}{(k+1)(k+2)} + \frac{2 m^{k+1}}{k+1} . \nonumber \\
\end{align}
\end{Lem}

\begin{Rem}
The constants 2 and 4 in (\ref{eq:sum_mmttk_bounds}) are likely not optimal
and are chosen only for convenience. 
\end{Rem}

\begin{proof}
If $a,b \in \Z$ with $a \leq b$ and $f : [a-1,b+1] \rightarrow \R$ is a continuous nondecreasing function, then
\begin{align}
\label{eq:fincreasing}
\int_{a-1}^b f(x) dx ~\leq~ \sum_{\tau = a}^b f(\tau) ~\leq~ \int_a^{b+1} f(x) dx.
\end{align}
Similarly, if $a,b \in \Z$ with $a \leq b$ and $f : [a-1,b+1] \rightarrow \R$ is a continuous nonincreasing function, then
\begin{align}
\label{eq:fdecreasing}
\int_{a-1}^b f(x) dx ~\geq~ \sum_{\tau = a}^b f(\tau) ~\geq~ \int_a^{b+1} f(x) dx.
\end{align}

The first pair of inequalities (\ref{eq:sum_tk_bounds}) is immediate from (\ref{eq:fincreasing}) since the function $x^k$ is 
increasing on $[0,m+1]$. To prove the second pair of inequalities (\ref{eq:sum_mmttk_bounds}) note that the function
$f(x) = (m-x)x^k$ is increasing on $[0, m (\frac{k}{k+1})]$ and decreasing on $[m (\frac{k}{k+1}), m+1]$ with
$\max_{x \in [0,m+1]} f(x) \equiv f_{\max} = f \left(m (\frac{k}{k+1}) \right) \leq \frac{m^{k+1}}{k+1}$. Thus, letting 
$A = \lfloor m(\frac{k}{k+1}) \rfloor$, $B = \lceil m(\frac{k}{k+1}) \rceil$ and applying the inequalities 
(\ref{eq:fincreasing}) and (\ref{eq:fdecreasing}), we have
\begin{align*}
\sum_{\tau = 1}^m f(\tau)
& \leq \sum_{\tau = 1}^{A-1} f(\tau) ~+~ \sum_{\tau = B+1}^m f(\tau) ~+~ 2 \cdot f_{\max}\\
& \leq \int_1^A f(x) dx ~+~ \int_B^m f(x) dx ~+~ 2 \cdot f_{\max} \\
& \leq \int_1^m f(x) dx ~+~ 2 \cdot f_{\max} \\
& \leq \frac{m^{k+2}}{(k+1)(k+2)} ~+~ \frac{2 m^{k+1}}{k+1} 
\end{align*}
and
\begin{align*}
\sum_{\tau = 1}^m f(\tau)
& \geq \sum_{\tau = 1}^{A-1} f(\tau) ~+~ \sum_{\tau = B+1}^m f(\tau) \\
& \geq \int_0^{A-1} f(x) dx ~+~ \int_{B+1}^{m+1} f(x) dx \\
& = \int_0^{m+1} f(x) dx ~-~ \int_{A-1}^{B+1} f(x) dx \\
& \geq \int_0^{m+1} f(x) dx ~-~ 3 \cdot f_{\max}  \\
& \geq \frac{(m+1)^{k+2}}{(k+1)(k+2)} ~-~ \frac{4 (m+1)^{k+1}}{k+1} \\ 
\end{align*}
for all $m \geq 2$. In the case $m=1$, the inequalities (\ref{eq:sum_mmttk_bounds}) may be verified directly.
\end{proof} 

\subsection{Estimates for $\E_{n,k}(I_t)$ and $Var_{n,k}(I_t)$}
\label{subsec:EstimateEandVarIt}
For a random variable $X$ taking values in $\{0,1, ... , m\}$,
\begin{align}
\label{eq:EXandEX2formulas}
\E(X) = \sum_{j = 0}^{m-1} \P(X > j) ~\mbox{ and }~ \E(X^2) = \E(X) + \sum_{j = 0}^{m-1} 2j \cdot \P(X > j).
\end{align}
Using these formulas along with Proposition \ref{prop:IndependenceAndDistributionItAndCtt} and Lemma 
\ref{lem:BasicSumEstimates}, we now obtain estimates for $\E_{n,k}(I_t)$, $\E_{n,k}(I_t^2)$, and $\Var_{n,k}(I_t)$. 

\begin{Cla}
\label{cla:EItEstimate}
Uniformly in $1 \leq t \leq n-1$ and $k \in \N$, 
\begin{align}
\label{eq:EItEstimate}
\E_{n,k}(I_t) = \frac{n-t}{k + 1} + O(1).
\end{align}
\end{Cla}

\begin{proof}
By (\ref{eq:DistributionIt}) and (\ref{eq:EXandEX2formulas}),
\begin{align*}
\E_{n,k} (I_t)  
=  \sum_{j=0}^{n-t-1} \left( \frac{n-t-j}{n-t+1} \right)^{k} 
= \frac{\sum_{\tau=1}^{n-t} \tau^{k} }{(n-t+1)^{k}}.
\end{align*}
Applying (\ref{eq:sum_tk_bounds}) gives
\begin{align}
\label{eq:EItUpperBound}
\E_{n,k} (I_t)  \leq \frac{(n-t+1)^{k+1}/(k+1)}{(n-t+1)^k} = \frac{n-t}{k+1} + \frac{1}{k+1}
\end{align} 
and
\begin{align*}
\E_{n,k} (I_t)  
& \geq \frac{(n-t)^{k+1}/(k+1)}{(n-t+1)^k} \\
& = \frac{n-t}{k+1} + \frac{n-t}{k+1} \left( \left(1 - \frac{1}{n-t+1}\right)^k - 1 \right) \\
& \geq \frac{n-t}{k+1} + \frac{n-t}{k+1} \left( \left(1 - \frac{k}{n-t+1}\right) - 1\right) \\
& \geq \frac{n-t}{k+1} - 1.
\end{align*} 
In the second to last inequality we use the fact that $\prod_{i=1}^k (1 - x_i) \geq 1 - \sum_{i=1}^k x_i$, 
for real numbers $x_1,...,x_k \in [0,1] $.
\end{proof}

\begin{Cla}
\label{cla:EIt2Estimate}
Uniformly in $1 \leq t \leq n-1$ and $k \in \N$, 
\begin{align}
\label{eq:EIt2Estimate}
\E_{n,k}(I_t^2) = \frac{2(n-t)^2}{(k + 1)(k+2)} + \frac{O(n)}{k}.
\end{align}
\end{Cla}

\begin{proof}
By (\ref{eq:DistributionIt}) and (\ref{eq:EXandEX2formulas}), 
\begin{align*}
\E_{n,k} \left(I_t^2 \right)
= \E_{n,k} (I_t) + \sum_{j=0}^{n-t-1} 2j \left( \frac{n-t-j}{n-t+1} \right)^{k} ~,
\end{align*}
and, by (\ref{eq:sum_mmttk_bounds}), we have the estimates
\begin{align*}
\sum_{j=0}^{n-t-1} 2j \left( \frac{n-t-j}{n-t+1} \right)^{k}
& = \frac{2}{(n-t+1)^{k}} \cdot \sum_{\tau=1}^{n-t} ( (n-t) - \tau )\tau^{k} \\
& \leq \frac{2}{(n-t+1)^{k}} \cdot \left[ \frac{(n-t)^{k + 2}}{(k+1)(k+2)} + \frac{2(n-t)^{k+1}}{k+1} \right] \\
& \leq \frac{2(n-t)^2}{(k+1)(k+2)} + \frac{4(n-t)}{k+1} ~,
\end{align*}
\begin{align*}
\sum_{j=0}^{n-t-1} 2j \left( \frac{n-t-j}{n-t+1} \right)^{k}
& = \frac{2}{(n-t+1)^{k}} \cdot \sum_{\tau=1}^{n-t} ( (n-t) - \tau )\tau^{k} \\
& \geq \frac{2}{(n-t+1)^{k}} \cdot \left[ \frac{(n-t+1)^{k + 2}}{(k+1)(k+2)} - \frac{4(n-t+1)^{k+1}}{k+1} \right] \\
& \geq \frac{2(n-t)^2}{(k+1)(k+2)} - \frac{8(n-t+1)}{k+1}.
\end{align*}
Using (\ref{eq:EItUpperBound}) and the fact that $I_t$ is nonnegative gives
\begin{align*}
\frac{2(n-t)^2}{(k+1)(k+2)} - \frac{8n}{k} \leq \E_{n,k} (I_t^2) \leq \frac{2(n-t)^2}{(k+1)(k+2)} + \frac{6n}{k}.
\end{align*} 
\end{proof}
 
\begin{Cla}
\label{cla:VarItEstimate}
Uniformly in $1 \leq t \leq n-1$ and $k \in \N$, 
\begin{align}
\label{eq:VarItEstimate}
\Var_{n,k}(I_t) = \frac{k(n-t)^2}{(k+1)^2(k+2)} + \frac{O(n)}{k} + O(1).
\end{align}
\end{Cla}

\begin{proof}
By Claims \ref{cla:EItEstimate} and \ref {cla:EIt2Estimate},
\begin{align*}
\Var_{n,k}(I_t) 
& = \left( \frac{2(n-t)^2}{(k + 1)(k+2)} + \frac{O(n)}{k} \right)  - \left( \frac{n-t}{k + 1} + O(1) \right)^2  \nonumber \\
& = \frac{k(n-t)^2}{(k+1)^2 (k+2)} + \frac{O(n)}{k} + O(1). 
\end{align*}
\end{proof} 

\subsection{The Case of Fixed $k$}
\label{subsec:FixedkCase}

Throughout Section \ref{subsec:FixedkCase} we assume $k \in \N$ is a fixed positive integer. 

\begin{proof}[Proof of Proposition \ref{prop:ExpAndVarI_Fixedk}]
By Claim \ref{cla:EItEstimate} and linearity of expectation
\begin{align*}
\E_{n,k}(I)
= \sum_{t=1}^{n-1} \left( \frac{n-t}{k+1} ~+~ O(1) \right) 
= \frac{n^2}{2(k+1)} ~+~ O(n).
\end{align*}
By Claim \ref{cla:VarItEstimate} and independence of the $I_t$'s
\begin{align*}
\Var_{n,k}(I)
& = \sum_{t=1}^{n-1} \left( \frac{k (n-t)^2}{(k+1)^2 (k+2)} ~+~ \frac{O(n)}{k} ~+~ O(1) \right) \\
& = \frac{k n^3}{3(k+1)^2 (k+2)} ~+~ O(n^2).
\end{align*}
\end{proof}

\begin{proof}[Proof of Theorem \ref{thm:WeakLawFixedk}]
By Proposition \ref{prop:ExpAndVarI_Fixedk} we have $\E_{n,k}(I/n^2) \rightarrow a_k$ and 
$\Var_{n,k}(I/n^2) \rightarrow 0$, as $n \rightarrow \infty$. Thus, the theorem follows by Chebyshev's inequality.  
\end{proof}

\begin{proof}[Proof of Theorem \ref{thm:CentralLimitTheoremFixedk}]

Run the $k$-card-minimum procedure independently on decks of each size $n = 1,2,...$ with given $k$. Denote by $\Int$
the random variable $I_t$ for the $n$ card deck, and by $\In$ the random variable $I$ for the $n$ card deck.  Also, 
denote the probability measure for this joint process simply by $\P$ and expectations and variances under
this measure simply by $\E(\cdot)$ and $\Var(\cdot)$. 

For $n \in \N$ and $1 \leq t \leq n$ define random variables $\Xnt$ by
\begin{align}
\label{eq:DefXnt}
\Xnt = \frac{ \Int - \E\left(\Int\right) }{ \sqrt{ \Var(\In) }}.
\end{align}
Note that the $\Xnt$'s are independent with zero mean, and moreover the following both hold. 
\begin{enumerate}
\item[(i)] For each $n$, $\sum_{t=1}^n \Var(\Xnt ) = \frac{1}{\Var(\In)} \cdot \sum_{t=1}^{n} \Var(\Int) = 1$, 
since $I_n^{(n)} \equiv 0$, and $I_1^{(n)},...,I_{n-1}^{(n)}$ are independent. 
\item[(ii)] With probability 1, uniformly in $t$, 
\begin{align*}
| \Xnt | \leq \frac{n}{\left(b_k n^3 + O(n^2)\right)^{1/2}} = O(1/n^{1/2}).
\end{align*}
So, for any $\epsilon > 0$, 
\begin{align*}
\lim_{n \to \infty}  \sum_{t=1}^n \E\left( \left| \Xnt \right|^2 ; \left| \Xnt \right| > \epsilon \right) = 0.
\end{align*}
\end{enumerate}
Thus, we may apply Theorem \ref{thm:LindbergFellerCLT} to conclude that
\begin{align*}
\frac{\In - \E \left( \In \right)}{\sqrt{\Var(\In)}} = \sum_{t=1}^n \Xnt \stackrel{d.}{\longrightarrow} Z, \mbox{ as } n \rightarrow \infty
\end{align*}
proving (\ref{eq:CentralLimitTheoremFixedk_1}). (\ref{eq:CentralLimitTheoremFixedk_2}) follows since 
\begin{align*}
\lim_{n \to \infty} \frac{\E \left( \In \right) - a_k n^2 }{\sqrt{\Var(\In)}} = 0 
~~\mbox{ and }~ \lim_{n \to \infty} \frac{ \sqrt{b_k} \cdot n^{3/2}}{ {\sqrt{\Var(\In)}} } = 1.
\end{align*}
\end{proof}

\subsection{The Case of Growing $k$}
\label{subsec:GrowingkCase}

Throughout Section \ref{subsec:GrowingkCase} we assume that $(k_n)_{n=1}^{\infty}$ is a nondecreasing sequence 
of positive integers such that $k_n \rightarrow \infty$ with $k_n = o(n)$. We will give proofs of Proposition 
\ref{prop:ExpAndVarI_Growingk} and Theorems \ref{thm:WeakLawGrowingk} and \ref{thm:CentralLimitTheoremGrowingk}.
The methods are very similar to those used in the previous section to establish the corresponding results 
in the case of fixed $k$. 

\begin{proof}[Proof of Proposition \ref{prop:ExpAndVarI_Growingk}] 
By Claim \ref{cla:EItEstimate} and linearity of expectation
\begin{align*}
\E_{n,k_n}(I)
= \sum_{t=1}^{n-1} \left( \frac{n-t}{k_n+1} ~+~ O(1) \right) 
= \frac{1}{2} \cdot \frac{n^2}{k_n} ~+~ o\left( \frac{n^2}{k_n}\right).
\end{align*}
By Claim \ref{cla:VarItEstimate} and independence of the $I_t$'s
\begin{align*}
\Var_{n,k_n}(I)
& = \sum_{t=1}^{n-1} \left( \frac{k_n (n-t)^2}{(k_n+1)^2 (k_n+2)} ~+~ \frac{O(n)}{k_n} ~+~ O(1) \right) \\
& = \frac{1}{3} \cdot \frac{n^3}{k_n^2} ~+~ o\left( \frac{n^3}{k_n^2}\right).
\end{align*}
\end{proof}

\begin{proof}[Proof of Theorem \ref{thm:WeakLawGrowingk}]
By Proposition \ref{prop:ExpAndVarI_Growingk} we have $\E_{n,k_n}\left(I \cdot \frac{k_n}{n^2}\right) \rightarrow 1/2$ and \\
$\Var_{n,k_n} \left(I \cdot \frac{k_n}{n^2}\right) \rightarrow 0$, as $n \rightarrow \infty$. Thus, the theorem follows by Chebyshev's inequality. 
\end{proof}

\begin{proof}[Proof of Theorem \ref{thm:CentralLimitTheoremGrowingk}]

Run the $k$-card-minimum procedure independently on decks of each size $n = 1,2,...$ with $k = k_n$ on the $n$ 
card deck. As in the proof of Theorem  \ref{thm:CentralLimitTheoremFixedk}, let $\In$ and $\Int$ be the random variables 
$I$ and $I_t$ for the $n$ card deck, and define $\Xnt$ by (\ref{eq:DefXnt}). Also, denote the probability measure for this joint process 
by $\P$ and expectations and variances under this measure by $\E(\cdot)$ and $\Var(\cdot)$. 

Again, the $\Xnt$'s are independent with zero mean and satisfy $\sum_{t=1}^n \Var(\Xnt) = 1$, for each $n$. 
So, the theorem will follow from Theorem \ref{thm:LindbergFellerCLT} if we can show that
\begin{align*}
\lim_{n \to \infty} \sum_{t=1}^n \E\left( \left| \Xnt \right|^2 ; \left| \Xnt \right| > \epsilon \right) = 0 ~,~ \mbox{ for any } \epsilon > 0.
\end{align*}

Now, by Claim \ref{cla:EItEstimate} and Proposition \ref{prop:ExpAndVarI_Growingk}, 
$\frac{\E (\Int )}{\sqrt{\Var(\In)}} = O(1/n^{1/2})$, uniformly in $t = 1, ... , n$.  
So, for any fixed $\epsilon > 0$, we know that for all sufficiently large $n$
\begin{align}
\label{eq:prob_lessnegepsilon_iszero}
\P \left( \Xnt < - \epsilon \right) = \P \left( \frac{\Int - \E ( \Int )}{\sqrt{ \Var(\In)}} < - \epsilon \right) = 0 ~,
\end{align}
for each $t =1, ..., n$. 

On the other hand, by Proposition \ref{prop:ExpAndVarI_Growingk} we also have the following upper bound 
on $\P \left( \Xnt > \epsilon \right)$ for all sufficiently large $n$ and each $1 \leq t \leq n$: 
\begin{align*}
\P \left( \Xnt > \epsilon \right) 
\leq \P \left( \frac{\Int}{\sqrt{\Var(\In)}} > \epsilon \right)
\leq \P \left( \Int > \epsilon \cdot \sqrt{n^3/(4k_n^2)} \right).
\end{align*}
Since $\Int$ takes only integer values between $0$ and $n-t$, it follows that $\P \left( \Xnt > \epsilon \right)$
can be nonzero only if $k_n \geq \frac{1}{2} \epsilon n^{3/2}/(n-t)$. In this case, the probability may be upper 
bounded as follows using Proposition \ref{prop:IndependenceAndDistributionItAndCtt}, for all $n$ sufficiently 
large that $\epsilon n^{3/2}/(2k_n) \geq 1$:
\begin{align} 
\label{eq:prob_greaternegepsilon_bound}
\P \left( \Xnt > \epsilon \right)  
& \leq  \P \left( \Int > \epsilon n^{3/2}/(2 k_n) \right) \nonumber \\
& = \left[ \left. \left( (n-t) - \lfloor \epsilon n^{3/2} / (2 k_n) \rfloor \right) \right/ (n-t+1) \right]^{k_n} \nonumber \\
& \leq \left[ \left. \left( n - \epsilon n^{3/2} / (4 k_n) \right) \right/ n \right]^{k_n} \nonumber \\
& = \left(1 - \frac{ (\epsilon/4) n^{1/2}}{k_n} \right)^{k_n} \nonumber \\
& \leq e^{ - (\epsilon/4)n^{1/2} }. 
\end{align}

Combining (\ref{eq:prob_lessnegepsilon_iszero}) and (\ref{eq:prob_greaternegepsilon_bound}) shows that for all sufficiently large $n$
\begin{align*}
\P \left(|\Xnt| > \epsilon \right) \leq  e^{ - (\epsilon/4)n^{1/2}} ~,~ 1 \leq t \leq n.
\end{align*}
To complete the proof, we note that $\left| \Int - \E \left( \Int \right) \right|$ can be at most $n$, for any $t$. 
So, by Proposition \ref{prop:ExpAndVarI_Growingk}, with probability 1 for all sufficiently large $n$ 
\begin{align*}
|\Xnt| \leq \frac{n}{\sqrt{\Var(\In)}} \leq n^{1/2} ~,~1 \leq t \leq n.
\end{align*}
Hence,
\begin{align*}
\lim_{n \to \infty} \sum_{t = 1}^{n} \E\left( \left| \Xnt \right|^2 ; \left| \Xnt \right| > \epsilon \right)
\leq \lim_{n \to \infty} n \cdot \left[ e^{ - (\epsilon/4)n^{1/2} } \cdot \left(n^{1/2}\right)^2 \right] = 0.
\end{align*}
\end{proof}
 
\section{Analysis of Longest Increasing Subsequence}
\label{sec:AnalysisOfLIS}

In this section we establish the results of Section \ref{subsec:ScalingOfL} for the length of the longest 
increasing subsequence in a $k$CM random permutation. An outline of the steps is as follows.

\begin{itemize}
\item In Section \ref{subsec:UpperBoundOnL} we establish a (high probability) upper bound on $L$.
The general method of proof is to divide the time set $[n] = \{1,...,n\}$ into blocks $B_i$ in an appropriate way,
and use Markov's inequality to upper bound the probability of having too long an increasing subsequence
in any time block. 
\item In Section \ref{subsec:LowerBoundOnL} we establish a (high probability) lower bound on $L$. 
The proof method is constructive, showing that a particular type of long enough increasing subsequence
will occur with high probability. 
\item In Section \ref{subsec:VarianceEstimateForL} we establish the variance estimate of Proposition \ref{prop:VarL}
using the Efron-Stein Inequality. 
\item Finally, in Section \ref{subsec:ScalinOfLAndWeakLaw} we prove Theorems \ref{thm:ScalingOfL} 
and \ref{thm:WeakLawL}, using Proposition \ref{prop:VarL} and the estimates of Sections \ref{subsec:UpperBoundOnL} 
and \ref{subsec:LowerBoundOnL}.
\end{itemize}

\subsection{Upper Bound on $L$}
\label{subsec:UpperBoundOnL}

In this section we prove the following proposition. 

\begin{Prop}
\label{prop:UpperBoundL}
Let $(k_n)_{n=1}^{\infty}$ be a sequence of positive integers satisfying $k_n = o(n)$. Then, for any 
$n \in \N$ and $\epsilon > 0$ sufficiently small that $4e(1+2\epsilon)\sqrt{k_n n} \leq n$, 
\begin{align*}
\P_{n,k_n}\left(L > 4e(1+2\epsilon)\sqrt{k_n n}\right) \leq \left( \frac{\log \sqrt{n/(\epsilon^2 k_n)}}{\log 4} \right)
\cdot \left(\frac{1}{1+\epsilon}\right)^{8e \sqrt{\epsilon}(1+\epsilon) k_n^{3/4} n^{1/4}}.
\end{align*}
\end{Prop}

Before proceeding to the proof, however, we must first introduce a bit more terminology and notation.
We say $s = ( (j_1,...,j_{\ell}), (t_1,...,t_{\ell}))$ is a \emph{time-indexed increasing subsequence} of [n]
if $1 \leq j_1 < ... < j_{\ell} \leq n$ and $1 \leq t_1 < ... < t_{\ell} \leq n$. Also, we say that $s$ is \emph{contained in} 
the random permutation $\sigma$ (written $s \subset \sigma$) if $C_{t_1} = j_1,..., C_{t_{\ell}} = j_{\ell}$.  Finally,
for a subset of times $A \subset [n]$, we define $S_{A,\ell,n}$ to be the set of all length-${\ell}$ time-indexed 
increasing subsequences with times $t_i \subset A$, and $N_{A,\ell,n}$ to be the (random) number of these that 
occur in $\sigma$.
\begin{align*}
S_{A,\ell,n} & = \{ s = ((j_1,...,j_{\ell}),(t_1,...,t_{\ell})) : 1 \leq j_1 < ... < j_{\ell} \leq n, t_1 < ... < t_{\ell}, t_i \in A, \forall i \}, \\
N_{A,\ell,n} & = | \{s \in S_{A,\ell,n} : s \subset \sigma \}|.
\end{align*}

The structure of the proof is as follows. We divide the time set $[n]$ into blocks $B_i$ according to a 4-adic splitting,
estimate $\E_{n,k_n}(N_{B_i,\ell,n})$ for each block $B_i$, and then use this estimate and Markov's inequality to show that, 
with high probability, the length $L_i$ of the longest increasing subsequence in the $i$-th block cannot be too large.
Hence, $L \leq \sum_i L_i$ also is not too large, with high probability. The details are given below.
 
\begin{proof}[Proof of Proposition \ref{prop:UpperBoundL}]

Fix any $n \in \N$ and $\epsilon > 0$ sufficiently small that $4e(1+2\epsilon)\sqrt{k_n n} \leq n$.
Note that the condition on $\epsilon$ implies $\frac{\log \sqrt{n/(\epsilon^2 k_n)}}{\log 4} \geq 1$.
Define $i_0 \in \N$ and the time blocks $B_i$, $i = 1,...,i_0$, by 
\begin{align*}
i_0  = \floor{ \frac{\log \sqrt{n/(\epsilon^2 k_n)}}{\log 4} } \mbox{ and }
B_i = \{n - \floor{n/4^{i-1}} + 1,...,n - \floor{n/4^i} \} ~,
\end{align*}
so that $B_1,...,B_{i_0}$ form a partition of the time set $\{1, ... , n - \floor{n/4^{i_0}} \}$.

The first piece of the proof is to bound $\E_{n,k_n}(N_{B_i,\ell,n})$, for each $i = 1,...,i_0$, 
which we do through a series of three steps as follows. 
\begin{enumerate}
\item At time $t$ there are $n-t+1$ cards in $D_t$ left to pick from. So, by the union bound, for each card $j$ 
and any choices $c_1,...,c_{t-1}$ for the first $t-1$ cards such that $c_{\tau} \not= j$, $\tau = 1,...,t-1$, we have
\begin{align}
\label{eq:SingleCardPlacementBound}
\P_{n,k_n}&(C_t = j|C_1 = c_1, ..., C_{t-1} = c_{t-1}) \nonumber \\
& \leq \P_{n,k_n}(\exists~ 1 \leq m \leq k_n : C_{t,m} = j|C_1 = c_1, ..., C_{t-1} = c_{t-1}) \nonumber \\
& \leq k_n/(n-t+1).
\end{align}
Hence, for any choices $c_1,...,c_{n - \floor{n/4^{i-1}}}$ of the first $n - \floor{n/4^{i-1}}$ cards
removed before time $n - \floor{n/4^{i-1}} + 1 = \min \{t : t \in B_i\}$ and any $s = ((j_1,...,j_{\ell}),(t_1,...,t_{\ell})) 
\in S_{B_i,\ell,n}$ such that $j_1,...,j_{\ell} \not\in \{c_1,...,c_{n - \floor{n/4^{i-1}}}\}$, we have
\begin{align}
\label{eq:Psinsigma}
& \P_{n,k_n}(s \subset \sigma| C_1 = c_1, ... , C_{n - \floor{n/4^{i-1}}} = c_{n - \floor{n/4^{i-1}}}) \nonumber \\
&= \prod_{m=1}^{\ell} \P_{n,k_n}\left(C_{t_m} = j_m | C_1 = c_1, ... , C_{n - \floor{n/4^{i-1}}} = c_{n - \floor{n/4^{i-1}}}, C_{t_1} = j_1,...,C_{t_{m-1}}=j_{m-1}\right) \nonumber \\
& \leq \prod_{m=1}^{\ell} \frac{k_n}{n-t_m+1}
\leq \left(\frac{k_n}{\floor{n/4^i} + 1} \right)^{\ell}
\leq  \left(\frac{k_n 4^i}{n} \right)^{\ell}.
\end{align}
\item For any particular choices $c_1,...,c_{n - \floor{n/4^{i-1}}}$ of the first $n - \floor{n/4^{i-1}}$ cards,
\begin{align}
\label{eq:NumberOf_sinsigma}
\left|\left\{s \in S_{B_i,\ell,n} : j_1,...,j_{\ell} \in d \right\} \right| \leq { \floor{n/4^{i-1}} \choose \ell }^2
\end{align}
where 
\begin{align*}
d = [n]/\{c_1,...,c_{n - \floor{n/4^{i-1}}} \} 
\end{align*}
is the remaining set of cards at time $n - \floor{n/4^{i-1}} + 1$ . The first factor of ${ \floor{n/4^{i-1}} \choose \ell }$ 
comes from possible choices for the cards $j_1,...,j_{\ell}$, and the second factor of ${ \floor{n/4^{i-1}} \choose \ell }$, 
which is an over estimate, comes from possible choices for the times $t_1,...,t_{\ell}$. 
\item Combining the estimates (\ref{eq:Psinsigma}) and (\ref{eq:NumberOf_sinsigma}) shows that for any particular 
choices $c_1,...,c_{n - \floor{n/4^{i-1}}}$ for the first $n - \floor{n/4^{i-1}}$ cards
\begin{align*}
\E_{n,k_n}&(N_{B_i,\ell,n} | C_1 = c_1, ... , C_{n - \floor{n/4^{i-1}}} = c_{n - \floor{n/4^{i-1}}}) \nonumber \\
& = \sum_{\{s \in S_{B_i,\ell,n} : j_1,...,j_{\ell} \in d \}} \hspace{-8 mm}
\P_{n,k_n} \left(s \subset \sigma| C_1 = c_1, ... , C_{n - \floor{n/4^{i-1}}} = c_{n - \floor{n/4^{i-1}}}\right) \\
&\leq \left(\frac{k_n 4^i}{n} \right)^{\ell} \cdot { \floor{n/4^{i-1}} \choose \ell }^2
\leq \left( \frac{e^2 k_n n }{4^{i-2} \ell^2} \right)^{\ell}.
\end{align*}
Hence, the same estimate also holds non-conditionally:
\begin{align}
\label{eq:ENBiln}
\E_{n,k_n}(N_{B_i,\ell,n}) \leq \left( \frac{e^2 k_n n }{4^{i-2} \ell^2} \right)^{\ell}.
\end{align} 
\end{enumerate}
Now, let $\ell_i = \ceil{4e(1+\epsilon) \sqrt{k_n n/4^i}}$, and let $L_i$ be the length 
of the longest increasing subsequence for cards in the time block $B_i$: 
\begin{align*}
L_i = \max \{\ell : \exists~ t_1,...,t_{\ell} \in B_i \mbox{ with } t_1 < ... < t_{\ell} \mbox{ and } C_{t_1} < ... < C_{t_{\ell}} \}.
\end{align*}
Then, applying Markov's inequality to (\ref{eq:ENBiln}) gives
\begin{align*}
\P_{n,k_n}(L_i \geq \ell_i) 
= \P_{n,k_n}(N_{B_i,\ell_i,n} \geq 1)
\leq \left( \frac{e^2 k_n n }{4^{i-2} \ell_i^2} \right)^{\ell_i}
\leq \left(\frac{1}{1+\epsilon}\right)^{2 \ell_i}.
\end{align*}
Further, by the definition of $i_0$, we know that for each $i = 1,...,i_0$,
\begin{align*}
\ell_i \geq 4e(1+\epsilon) \sqrt{k_n n/4^{i_0}} \geq 4e(1+\epsilon) \sqrt{\epsilon} k_n^{3/4} n^{1/4}.
\end{align*}
Thus,
\begin{align*}
\P_{n,k_n}(\exists~ 1 \leq i \leq i_0 : L_i \geq \ell_i) \leq i_0 \cdot \left(\frac{1}{1+\epsilon}\right)^{8e(1+\epsilon) \sqrt{\epsilon} k_n^{3/4} n^{1/4} }.
\end{align*}
The claim follows, since on the event $\{L_i < \ell_i , i = 1,...,i_0 \}$, we have
\begin{align*}
L & \leq \sum_{i=1}^{i_0} (\ell_i - 1) ~+~ \floor{n/4^{i_0}} \\
& < \sum_{i=1}^{\infty} 4e(1+\epsilon) \sqrt{k_n n/4^i} ~~+~~ 4\epsilon \sqrt{k_n n} \\
& < 4e(1+2\epsilon) \sqrt{k_n n}.
\end{align*} 
\end{proof}

\begin{Rems}~
\begin{enumerate}
\item One may consider partitioning the time set $[n]$ with an $x$-adic splitting, for any $x > 1$, rather than specifically with the 4-adic 
splitting we use. That is, one may replace $4$ by $x$ in the definition of $B_i$. Doing this for general $x$, using estimates as above, 
gives an upper bound on $L$ of roughly $\frac{ex}{\sqrt{x} - 1} \cdot \sqrt{k_n n}$, up to $\epsilon$ corrections. This bound is minimized 
by taking $x = 4$.
\item A more straightforward approach would be to not partition the time set into blocks at all, and simply bound 
the expected total number of length-$\ell$ time-indexed increasing subsequences occurring in $\sigma$ by 
\begin{align*}
\E_{n,k_n}(N_{\ell,n}) \leq |S_{\ell,n}| \cdot \max_{s \in S_{\ell,n}} \P_{n,k_n}(s \subset \sigma)
\end{align*}
where $S_{\ell,n}$ is the set of all length-$\ell$ time-indexed increasing subsequences of $[n]$. However, this 
does not work as easily, because the bound (\ref{eq:SingleCardPlacementBound}) is not good if $t$ is too large,
and, therefore, obtaining a good bound on $\P_{n,k_n}(s \subset \sigma)$ for an arbitrary time-indexed increasing 
subsequence $s = ((j_1,...,j_{\ell}),(t_1,...,t_{\ell}))$, without any constraint on the times $t_1,...,t_{\ell}$, is more difficult. 
\end{enumerate}
\end{Rems}

\subsection{Lower Bound on $L$}
\label{subsec:LowerBoundOnL}

In this section we prove the following proposition. 

\begin{Prop}
\label{prop:LowerBoundL}
Let $(k_n)_{n=1}^{\infty}$ be a sequence of positive integers satisfying $k_n = o(n)$. 
Then, for any $0 < \epsilon < 1/2$ there exists $n_0 \in \N$ such that
\begin{align}
\label{eq:LowerBoundL}
\P_{n,k_n}\left(L < (1/2 - \epsilon) \sqrt{k_n n}\right) \leq \exp\left(-\frac{\epsilon^2}{4(1-\epsilon)} \cdot \sqrt{k_n n}\right)
\end{align}
for all $n \geq n_0$. 
\end{Prop}

\begin{proof}
Throughout $0 < \epsilon < 1/2$ is fixed and $n_0 = n_0(\epsilon)$ is chosen sufficiently large that for all $n \geq n_0$,
\begin{align*}
\frac{n/2}{\ceil{\sqrt{n/k_n}}} - 1 \geq (1/2 - \epsilon) \sqrt{k_n n} ~\mbox{ and }~
e^{-\sqrt{k_n/n}} \leq1 - (1-\epsilon) \sqrt{k_n/n}.
\end{align*}
Our proof is based upon the constructive procedure given below. 
\begin{itemize} 
\item Let $T_0 = 0$ and $R_0^+ = [n]$.
\item Then, for $m = 1,2, ...$ :
\begin{itemize}
\item[*] Let $S_m$ be the set consisting of the lowest $\ceil{\sqrt{n/k_n}}$ cards in $R_{m-1}^+$.
\item[*] Let $T_m = \min\{t > T_{m-1} :C_t \in S_m\}$ be the first time some card in the next target set $S_m$ is picked. 
\item[*] Let $R_m^+ = \{j \in D_{T_m + 1}: j > C_{T_m}\}$ be the set of cards remaining in the deck after time $T_m$,
which are larger than card $C_{T_m}$.  
\end{itemize}
\item Continue in this fashion until the first time $m$, such that there are fewer than $\ceil{\sqrt{n/k_n}}$
cards in the remaining set $R_m^+$. That is, $R_m^+$, $S_m$, and $T_m$ are defined inductively by 
the above relations for $m = 1,...,M$ where 
\begin{align*}
M = \min \left\{m : |R_m^+| < \ceil{\sqrt{n/k_n}}\right\}. 
\end{align*}
\end{itemize} 

With this construction we have $C_{T_{m+1}} > C_{T_m}$, for each $1 \leq m < M$, so 
the random sequence $C_{T_1}, ... , C_{T_M}$ is an increasing subsequence in $\sigma$. 
Therefore, it suffices to show that for all $n \geq n_0$,
\begin{align}
\label{eq:ETS_MBig}
\P_{n,k_n}\left(M < (1/2 - \epsilon) \sqrt{k_n n}\right) \leq \exp\left(-\frac{\epsilon^2}{4(1-\epsilon)} \cdot \sqrt{k_n n}\right).
\end{align}
To this end, we define the following additional random variables.
\begin{itemize}
\item $1\leq A_m \leq \ceil{\sqrt{n/k_n}}$ is the relative position of card $C_{T_m}$ in the $m$-th target 
interval $S_m$, and $B_m$ is the number of cards greater than $C_{T_{m-1}}$ which 
are removed from the deck between times $T_{m-1}$ and $T_m$:
\begin{align*}
A_m = |\{j \in S_m : C_{T_m} \geq j \}| ~\mbox{ and }~ 
B_m = |\{T_{m-1} < t < T_m : C_t > C_{T_{m-1}} \}|
\end{align*}
where $C_{T_0} = C_0 \equiv 0$. 
\item $m(t)$ is the index of the most recent stopping time $T_m$, 
and $\TM$ is the set of times $t \leq T_M$ at which the random card chosen, $C_t$, 
is greater than the card chosen at the most recent stopping time:
\begin{align*}
m(t) = \max \{0 \leq m \leq M : T_m < t\} ~\mbox{ and } \TM = \{ t \leq T_M : C_t > C_{T_{m(t)}}\}.
\end{align*}
\item $N = |\TM|$ and, for $1 \leq i \leq N$, $t_i$ is the $i$-th lowest element of $\TM$.
\end{itemize} 
We observe that:
\begin{itemize}
\item[(i)] $\sum_{m=1}^M A_m \leq M \ceil{\sqrt{n/k_n}}$.
\item[(ii)] $\sum_{m=1}^M B_m = N - M \leq N - 1$.
\item[(iii)] $|R_m^+| = n - \sum_{i=1}^m A_m - \sum_{i=1}^m B_m$, for each $1 \leq m \leq M$.
So, in particular, $\sum_{m=1}^{M} A_m + \sum_{m=1}^M B_m > n - \ceil{\sqrt{n/k_n}}$.
\end{itemize}
Points (i) and (ii) follow directly from the definitions, and (iii) is easily shown by induction on $m$. 
Our proof is based on these simple facts and the following claim. \\ \\
\underline{Claim}: For $n \geq n_0$, 
\begin{align}
\label{eq:TheClaim}
\P_{n,k_n}\left(N \geq \ceil{n/2} , M < (1/2-\epsilon)\sqrt{k_n n}\right) 
\leq \exp\left(-\frac{\epsilon^2}{4(1-\epsilon)} \cdot \sqrt{k_n n}\right).
\end{align}

If $N < \ceil{n/2}$, then (ii) and (iii) imply that $\sum_{m=1}^M A_m > n/2 - \ceil{\sqrt{n/k_n}}$, which, in 
turn, implies $M > \frac{n/2}{\ceil{\sqrt{n/k_n}}} - 1 \geq (1/2 - \epsilon)\sqrt{k_n n}$, for all $n \geq n_0$, by (i). 
Thus, if (\ref{eq:TheClaim}) holds so does (\ref{eq:ETS_MBig}). So, it remains only to show (\ref{eq:TheClaim}).
This we do through a series of four steps below. 

\begin{enumerate}

\item The event $\{m(t) < M\}$ depends only on $C_1,...,C_{t-1}$, and conditioned on $m(t) < M$ and $C_t > C_{T_{m(t)}}$ the $k_n$ 
random card choices $C_{t,1},...,C_{t,k_n}$ are i.i.d. uniform on $D_t^+ \equiv \{j \in D_t:j > C_{T_{m(t)}}\}$. Thus, if $c_1,...,c_{t-1}$ 
are any particular choices for the first $t-1$ cards such that $C_1 = c_1,...,C_{t-1} = c_{t-1}$ implies $m(t) < M$, we have
\begin{align*}
\P_{n,k_n}&(C_t \not\in S_{m(t)+1} |C_1= c_1,...,C_{t-1} = c_{t-1}, C_t > C_{T_{m(t)}}) \\
& = \left( \frac{d - \ceil{\sqrt{n/k_n}}}{d} \right)^{k_n}
\leq \left( \frac{n - \sqrt{n/k_n}}{n} \right)^{k_n}
\leq e^{-\sqrt{k_n/n}}
\end{align*} 
where $\ceil{\sqrt{n/k_n}} \leq d \leq n$ is the size of the set $D_t^+$ on the event $\{C_1 = c_1,...,C_{t-1} = c_{t-1}\}$.  
Hence, for all $n \geq n_0$ and any such $c_1,...,c_{t-1}$, we have
\begin{align}
\label{eq:GoodProbHitSmnext}
\P_{n,k_n}(C_t \in S_{m(t)+1} |C_1= c_1,...,C_{t-1} = c_{t-1}, C_t > C_{T_{m(t)}}) \geq (1-\epsilon) \sqrt{k_n/n}. 
\end{align}

\item Let $(Z_i)_{i \in \N}$ be defined by
\begin{align*}
Z_i =
\left\{ \begin{array}{l}
\indicator \{C_{t_i} \in S_{m(t_i) + 1} \}, \mbox{ if } i \leq N \\
0, \mbox{ otherwise}.
\end{array} \right.
\end{align*}
Then, by (\ref{eq:GoodProbHitSmnext}), for any particular values $z_1,...,z_{i-1}$ of the random variables 
$Z_1,...,Z_{i-1}$ such that the event $\{Z_1 = z_1, ..., Z_{i-1} = z_{i-1}, N \geq i\}$ is possible we have
\begin{align}
\label{eq:GoodProbZiEquals1}
\P_{n,k_n}(Z_i = 1| Z_1 = z_1,...,Z_{i-1} = z_{i-1}, N \geq i) \geq (1-\epsilon) \sqrt{k_n/n}.
\end{align}

\item Let $p_{\epsilon,n} = (1-\epsilon) \sqrt{k_n/n}$. Then, by (\ref{eq:GoodProbZiEquals1}), it is possible to 
couple a sequence of i.i.d. Ber($p_{\epsilon,n}$) random variables, $(X_i)_{i \in \N}$, to the $k$CM process
such that, for all $1 \leq i \leq N$, $Z_i = 1$ whenever $X_i = 1$. That is, by enlarging the underlying probability
space for the $k$CM process $(C_t)_{t=1}^n$, we may define this process along with the i.i.d. Bernoulli
sequence $(X_i)$ on a common probability space $\Omega$, such that $Z_i = 1$, for all $1 \leq i \leq N$ 
with $X_i = 1$. The measure for this joint space will, with a slight abuse of notation, continue be denoted $\P_{n,k_n}$.

\item If $X$ has Bin$(m,p)$ distribution, then by \cite[Theorem A.1.13]{Alon2000},
\begin{align*}
\P(X < mp - x) \leq e^{-\left(\frac{x^2}{2mp}\right)}~,~ \forall x \geq 0.
\end{align*}  
Taking $X = \sum_{i = 1}^{\ceil{n/2}} X_i$ and $x = \epsilon \sqrt{k_n/n} \ceil{n/2}$ gives
\begin{align*}
\P_{n,k_n}&\left(\sum_{i=1}^{\ceil{n/2}} X_i < (1/2 - \epsilon) \sqrt{k_n n} \right) \\
&\leq \P_{n,k_n}\left(\sum_{i=1}^{\ceil{n/2}} X_i < (1 - \epsilon) \sqrt{k_n/n} \ceil{n/2} - \epsilon \sqrt{k_n/n} \ceil{n/2} \right) \\
& \leq \exp\left( -\frac{\left(\epsilon \sqrt{k_n/n} \ceil{n/2}\right)^2}{2 \ceil{n/2}(1 -\epsilon) \sqrt{k_n/n}} \right)
\leq \exp\left(-\frac{\epsilon^2}{4(1-\epsilon)} \cdot \sqrt{k_n n}\right).
\end{align*}
Thus, since $M = \sum_{i = 1}^N Z_i$, the coupling between the $X_i$'s and $Z_i$'s implies
\begin{align*}
\P_{n,k_n}&\left(N \geq \ceil{n/2} , M < (1/2-\epsilon)\sqrt{k_n n}\right) \\
& = \P_{n,k_n}\left(N \geq \ceil{n/2} , \sum_{i=1}^{N} Z_i < (1/2-\epsilon)\sqrt{k_n n}\right) \\
& \leq \P_{n,k_n}\left(N \geq \ceil{n/2} , \sum_{i=1}^{\ceil{n/2}} Z_i < (1/2-\epsilon)\sqrt{k_n n}\right) \\
& \leq \P_{n,k_n}\left(N \geq \ceil{n/2} , \sum_{i=1}^{\ceil{n/2}} X_i < (1/2-\epsilon)\sqrt{k_n n}\right) \\
& \leq \P_{n,k_n}\left(\sum_{i=1}^{\ceil{n/2}} X_i < (1/2-\epsilon)\sqrt{k_n n}\right) \\
& \leq \exp\left(-\frac{\epsilon^2}{4(1-\epsilon)} \cdot \sqrt{k_n n}\right) 
\end{align*}
proving (\ref{eq:TheClaim}).

\end{enumerate}
\end{proof}  

\subsection{Variance Estimate for $L$}
\label{subsec:VarianceEstimateForL}

In this section we prove Proposition \ref{prop:VarL}. For this we will need the following two lemmas. 

\begin{Lem}[Efron-Stein Inequality]
\label{lem:EfronStein}
Let $X_1,...,X_n$ be independent random variables and let $Y=f(X_1,...,X_n)$, for some
measuable function $f : \R^n \rightarrow \R$. Then
\begin{align*}
\Var(Y) \leq \sum_{t=1}^n \E\left\{ \Var(Y|X_1,...,X_{t-1},X_{t+1},...,X_n) \right\}.
\end{align*}
\end{Lem}

\begin{Lem}
\label{lem:ChangeLBy1}
Let $j_1,...,j_n \in [n] \times ... \times [1]$ be any sequence of possible choices for the relative card positions 
$\Ct_1,...,\Ct_n$ and let $i_1,...,i_n$ be the associated choices of actual cards $C_1,...,C_n$. That is,
\begin{align*}
\Ct_1 = j_1,..., \Ct_n = j_n \iff C_1 = i_1,..., C_n = i_n.
\end{align*}
Also, let $t \in [n]$, let $j_t' \not= j_t$ be another possible choice for the relative position of the $t$-th card,
and let $i'_1,...,i'_n$ be the associated sequence of cards $C_1,...,C_n$ obtained with relative
choices $\Ct_t = j_t'$ and $\Ct_{\tau} = j_{\tau}$, for all $\tau \not= t$. That is,
\begin{align*}
\Ct_t = j_t' \mbox{ and } \Ct_{\tau} = j_{\tau}, \forall \tau \not= t  \iff &~ C_1 = i_1' , ... , C_n = i_n'.
\end{align*}
Finally, let $\ell$ and $\ell'$ denote, respectively, the lengths of the longest increasing subsequences in 
the permutations $\sigma = (i_1,...,i_n)$ and $\sigma' = (i_1', ..., i_n')$. Then
\begin{align*}
|\ell - \ell'| \leq 1.
\end{align*}
\end{Lem}

Using these lemmas the proof of the proposition is actually quite simple. So, we present this first, followed by 
the more involved proof of Lemma \ref{lem:ChangeLBy1}, which requires analysis of several different cases. 
For a proof of Lemma \ref{lem:EfronStein}, see \cite{Boucheron2004}. 

\begin{proof}[Proof of Proposition \ref{prop:VarL}]
$L$ is a (deterministic) function of the relative card choices $\Ct_1,...,\Ct_n$, and by Proposition 
\ref{prop:IndependenceAndDistributionItAndCtt} these relative card choices are independent.  
So, by Lemma \ref{lem:EfronStein}, we have
\begin{align}
\label{eq:EfronSteinVarSumBound}
\Var_{n,k}(L) \leq \sum_{t=1}^n \E_{n,k} \left\{ \Var_{n,k}(L|\Ct_1,...,\Ct_{t-1},\Ct_{t+1},...,\Ct_n) \right\}.
\end{align}
Moreover, by Lemma \ref{lem:ChangeLBy1}, $L$ can take only one of two possible consecutive integer values
if $\Ct_1,...,\Ct_{t-1},\Ct_{t+1},...,\Ct_n$ are fixed. So, for each $t$, 
\begin{align}
\label{eq:VarConditionedOnEachBound}
\E_{n,k} \{\Var_{n,k}(L|\Ct_1,...,\Ct_{t-1},\Ct_{t+1},...,\Ct_n)\} \leq \max_{p \in [0,1]} \Var\left\{ \mbox{Ber}(p) \right\} = 1/4.
\end{align}
Together (\ref{eq:EfronSteinVarSumBound}) and (\ref{eq:VarConditionedOnEachBound}) imply the claim.
\end{proof}

\begin{proof}[Proof of Lemma \ref{lem:ChangeLBy1}]
For the proof we will need the following additional notation.
\begin{itemize}
\item $\ell_t$ and $\ell_t'$ denote, respectively, the lengths of the longest increasing subsequences in 
$\sigma_t = (i_1,...,i_{t-1}, i_{t+1},...,i_n)$ and $\sigma_t' = (i'_1,...,i'_{t-1},i_{t+1}',...,i_n')$. 
\item For each $\tau \in [n]$, $d_{\tau} = [n]/\{i_1,...,i_{\tau-1}\}$ and $d_{\tau}' = [n]/\{i'_1,...,i'_{\tau -1}\}$. 
Also, $d_{{\tau},j}$ is $j$-th lowest numbered element in the set $d_{\tau}$, and $d'_{{\tau},j}$ is $j$-th 
lowest numbered element in the set $d'_{\tau}$.
\item Finally, for convenience, we write $\ell \sim (i_{t_1}, ..., i_{t_\ell})$ to mean than $(i_{t_1}, ... , i_{t_\ell})$ is a longest increasing
subsequence in $\sigma$ (i.e. $t_1 < ... < t_{\ell}$ and $i_{t_1} < ... < i_{t_\ell}$). Similar notation is also used with $\ell'$, $\ell_t$, and $\ell'_t$. 
\end{itemize}
From the definitions, it is immediate that:
\begin{itemize}
\item[(i)] $\ell \in \{\ell_t, \ell_t + 1\}$ and $\ell' \in \{\ell'_t, \ell'_t + 1\}$.
\item[(ii)] $i_{\tau} = i'_{\tau}$, for each $\tau = 1,...,t-1$. 
\item[(iii)] The position of card $d_{t+1,j}$ in $\sigma$ is the same as the position of card 
$d'_{t+1,j}$ in $\sigma'$, for each $j = 1,...,n-t$. That is, for any $\tau > t$, $i_{\tau} = d_{t+1,j}$ if and only if 
$i'_{\tau} = d'_{t+1,j}$. Hence, for any sequence of times $t < t_1 < ... < t_m \leq n$,
$i_{t_1} < ... < i_{t_m}$ if and only if $i'_{t_1} < ... < i'_{t_m}$. 

\end{itemize}

In the remainder of the proof we will assume, without out loss of generality, that $j_t' > j_t$. Under this 
assumption we have also the following relations between the elements of $d_t$, $d_{t+1}$, and $d'_{t+1}$. 
\begin{itemize}
\item[(iv)] ~ \vspace{- 3 mm}
\begin{align*}
\mbox{ For } & j = 1,...,j_t - 1 ~,~~ d_{t+1,j} = d'_{t+1,j} = d_{t,j}. \\
\mbox{ For } & j = j_t,...,j_t' - 1 ~,~ d_{t+1,j} = d_{t,j+1}. \\
\mbox{ For } & j = j_t,...,j_t' - 1 ~,~ d'_{t+1,j} = d_{t,j}. \\
\mbox{ For } & j = j'_t,...,n-t ~,~~ d_{t+1,j} = d'_{t+1,j} = d_{t,j+1}.
\end{align*} 
In particular, $d_{t+1,j} \geq d'_{t+1,j}$ for each $j = 1,...,n-t$. So, $i_{\tau} \geq i'_{\tau}$, for all $\tau > t$, by (iii). \\
\end{itemize}
Using facts (i)-(iv) we now prove the lemma through a series of two claims. \\

\noindent
\underline{Claim 1:} $\ell' \leq \ell + 1$. 

\noindent
\underline{Pf:} Let $\ell'_t \sim (i'_{t_1},...,i'_{t_{\ell'_t}})$ and let $M = \max \{m : t_m < t\}$, with the convention $M = 0$ 
if there is no such $m$. If $M = \ell'_t$ then $i_{t_1} < ... < i_{t_{\ell'_t}}$ by (ii), and if $M = 0$ then $i_{t_1} < ... < i_{t_{\ell'_t}}$ 
by (iii). If $0 < M < \ell'_t$ then $i_{t_1} < ... < i_{t_M}$ by (ii), $i_{t_{M+1}} < ... < i_{t_{\ell'_t}}$ by (iii), and $i_{t_{M+1}} \geq
i'_{t_{M+1}} > i'_{t_M} = i_{t_M}$ by (iv). Thus, again, $i_{t_1} < ... < i_{t_{\ell'_t}}$. It follows that $\ell_t \geq \ell'_t$, 
from which the claim follows by (i). \\

\noindent
\underline{Claim 2:} $\ell \leq \ell' + 1$. \\
\underline{Pf:} Let $\ell \sim (i_{t_1},...,i_{t_\ell})$ and let $M = \max \{m : t_m < t\}$, with the convention $M = 0$ 
if there is no such $m$. We consider separately two cases: $M \in \{0, \ell - 1, \ell\}$ and  $0< M < \ell - 1$. \\

\noindent
\underline{Case 1:} $M \in \{0, \ell -1, \ell\}$.\\
If $M = \ell$ or $M = \ell - 1$, then $i'_{t_1} < ... < i'_{t_{\ell-1}}$ by (ii), which implies $\ell' \geq \ell - 1$. 
If $M = 0$ then it is possible $t_1 = t$, but $t_2,..., t_{\ell}$ are all greater than $t$. Thus, by (iii), 
$i'_{t_2} < ... < i'_{t_\ell}$, which implies $\ell' \geq \ell - 1$. \\

\noindent
\underline{Case 2:} $0 < M < \ell - 1$ \\
In this case, $i_{t_{M+1}} = d_{t,a}$ for some $a$ and $i_{t_{M+2}} = d_{t,b}$ for some $b > a$, and it follows from (iii) and (iv) that  
$i'_{t_{M+2}} = d_{t,b'}$ for some $b' \geq b-1$. Thus, we have $i'_{t_{M+2}} = d_{t,b'} \geq d_{t,a} = i_{t_M+1} > i_{t_M} = i'_{t_M}$. 
Also, $i'_{t_1} < ... < i'_{t_M}$ by (ii), and $i'_{t_{M+2}} < ... < i'_{t_\ell}$ by (iii). So, altogether, we have 
$i'_{t_1} < ... < i'_{t_M} < i'_{t_{M+2}} < ... < i'_{t_{\ell}}$, which implies $\ell' \geq \ell - 1$. 

\end{proof}

\subsection{Scaling of $L$ and the Weak Law}
\label{subsec:ScalinOfLAndWeakLaw} 

In this section we prove Theorems \ref{thm:ScalingOfL} and \ref{thm:WeakLawL}.

\begin{proof}[Proof of Theorem \ref{thm:ScalingOfL}]
(\ref{eq:ScalingOfL}) is immediate from Propositions \ref{prop:UpperBoundL} and \ref{prop:LowerBoundL},
and the lower bound in (\ref{eq:ScalingOfL}) gives 
\begin{align*}
1/2 \leq \liminf_{n \to \infty} ~ \E_{n,k_n}(L)/\sqrt{k_n n}.
\end{align*}
Thus, it remains only to show 
\begin{align}
\label{eq:limsupELbound}
\limsup_{n \to \infty} ~ \E_{n,k_n}(L)/\sqrt{k_n n} \leq 4e.
\end{align}
To do this we will use the upper bound of Proposition \ref{prop:UpperBoundL}.

For $n \in \N$, let $y_{\max} = y_{\max}(n)$ be defined by $4e(1+2y_{\max})\sqrt{k_n n} = n$, and,
for $\epsilon > 0$, let $c_{\epsilon} > 0$ be defined by $e^{-c_{\epsilon}} = \left(\frac{1}{1+\epsilon}\right)^{8e}$. 
Then, by Proposition \ref{prop:UpperBoundL}, we know that for any $\epsilon > 0$ the following
estimate holds for all sufficiently large $n$ and each $\epsilon \leq y \leq y_{\max}(n)$:
\begin{align*}
\P_{n,k_n}\left(L > 4e(1+2y) \sqrt{k_n n}\right)
& \leq \left( \frac{\log \sqrt{n/(y^2 k_n)}}{\log 4} \right) \cdot \left(\frac{1}{1+y}\right)^{8e \sqrt{y} (1+y) k_n^{3/4} n^{1/4} } \\
& \leq \log(n) \cdot e^{-c_{\epsilon} n^{1/4} y}.
\end{align*}
Thus, using the change of variables $x = 4e(1+2y) \sqrt{k_n n}$, we have
\begin{align*}
\E_{n,k_n}(L) 
& = \int_0^{n} \P_{n,k_n}(L > x) dx \\
& \leq 4e(1+2\epsilon) \sqrt{k_n n} ~+~ \int_{4e(1+2\epsilon)\sqrt{k_n n}}^{n} ~ \P_{n,k_n}(L > x) dx \\
& \leq 4e(1+2\epsilon) \sqrt{k_n n} ~+~ 8e \sqrt{k_n n} \int_{\epsilon}^{y_{\max}} \log(n) \cdot e^{-c_{\epsilon} n^{1/4} y} ~dy \\
& \leq 4e(1+2\epsilon) \sqrt{k_n n} ~+~ 8e \sqrt{k_n n} \cdot \log(n) \cdot \frac{e^{-c_{\epsilon} n^{1/4} \epsilon}}{c_{\epsilon} n^{1/4}}
\end{align*} 
for all sufficiently large $n$. So, 
\begin{align*}
\limsup_{n \to \infty} ~ \E_{n,k_n}(L)/\sqrt{k_n n} \leq 4e(1+2\epsilon).
\end{align*}
Since $\epsilon > 0$ is arbitrary this shows that (\ref{eq:limsupELbound}) holds, completing the proof.
\end{proof} 

\begin{proof}[Proof of Theorem \ref{thm:WeakLawL}]
By Theorem \ref{thm:ScalingOfL}, $\E_{n,k_n}(L) \geq \frac{1}{4} \sqrt{k_n n}$ for all 
sufficiently large $n$. So, by Chebyshev's inequality and Proposition \ref{prop:VarL}, we have 
\begin{align}
\label{eq:LFluctuationSmall}
\P_{n,k_n}\left( \left| \frac{L}{\E_{n,k_n}(L)} - 1 \right| > \epsilon \right) 
\leq \frac{\Var_{n,k_n}(L)}{\epsilon^2 \cdot \left( \E_{n,k_n}(L) \right)^2}
\leq \frac{4}{\epsilon^2 k_n}
\end{align}
for all sufficiently large $n$. The theorem follows, since the right hand side of 
(\ref{eq:LFluctuationSmall}) tends to $0$, if $k_n \rightarrow \infty$. 
\end{proof}

\section{Analysis of Optimality for the $k$-Card-Minimum Procedure}
\label{sec:AnalysisOfOptimality_kCMProcedure} 

In this section we prove Propositions \ref{eq:MinStrategyOptimalForI} and \ref{eq:MinStrategyNotOptimalForL}. 
Random variables and probability measures are defined as above for the $k$CM procedure, 
but with superscripts to indicate the strategy used as needed. 

\begin{proof}[Proof of Proposition \ref{eq:MinStrategyOptimalForI}]
Fix $k,n \in \N$ and let $\SM_{other} \not= \SM_{\min}$ be any other $(k,n)$ choice strategy. Couple the $\SM_{\min}$ and $\SM_{other}$ 
processes so that in both processes the $k$ random card choices $C_{t,1},...,C_{t,k}$ at each time $t$ occupy the same relative positions 
in the remaining sets $D_t$. That is, for each $1 \leq t \leq n$ and $1 \leq i \leq k$,
\begin{align*}
C_{t,i}^{\min} = \mbox{ $j$-th lowest card in } D_t^{\min} \iff C_{t,i}^{other} = \mbox{ $j$-th lowest card in } D_t^{other}.
\end{align*}
Then, under this coupling, we will have $\Ct_t^{\min} \leq \Ct_t^{other}$ with probability 1, for each $t$. Hence, also,
$I_t^{\min} \leq I_t^{other}$ with probability 1, and $I^{\min} \leq I^{other}$ with probability 1, by relations 
(\ref{eq:ItCttRelation}) and (\ref{eq:IequalssumIts}). The claim follows. 
\end{proof} 
 
\begin{proof}[Proof of Proposition \ref{eq:MinStrategyNotOptimalForL}]
We show separately the two claims of the proposition. \\

\noindent
\underline{Claim 1}: For any $k \geq 2$ and $n \geq 5$ there is no $(k,n)$ stochastically optimal strategy for maximizing $L$. \\ \\
\underline{Pf}:
Fix $k \geq 2$, $n \geq 5$ and let $\ell = \floor{n/2}$, $m = \ceil{n/2}$.  Assume there exists a $(k,n)$ optimal strategy $\SM$,
and consider the event $A$ that:
\begin{enumerate}
\item $C_1 = \ell + 1, C_2 = \ell + 2, ... , C_{m-2} = \ell + m - 2 = n - 2$, and
\item $C_{m-1,1} = 1$ and $C_{m-1,j} = n-1$, for each $2 \leq j \leq k$. 
\end{enumerate}
There are two possibilities on the event $A$. Either card $1$ or card $n-1$ is selected with the strategy $\SM$ for $C_{m-1}$. 
In either case, we define the strategy $\hat{\SM}$ as follows:
\begin{itemize}
\item $\hat{\SM}$ uses exactly the same selection rules as $\SM$ for all $1 \leq t \leq m-2$.
\item If $A$ does not occur, $\hat{\SM}$ also uses the same selection rules as $\SM$ for all $t \geq m-1$.
\item If $A$ does occur, then $\hat{\SM}$ makes the opposite selection as $\SM$ for card $C_{m-1}$, 
and then uses the choices of the minimum strategy $\SM_{\min}$ for all $t \geq m$. 
\end{itemize} 
We consider separately two cases, depending on which of the two possible choices the strategy $\SM$ selects 
for $C_{m-1}$ on the event $A$. In each case, we will show that
\begin{align*}
\P_{n,k}^{\hat{\SM}}(L \geq x) > \P_{n,k}^{\SM}(L \geq x), ~\mbox{ for some}~ x,
\end{align*} 
which contradicts the fact that $\SM$ is stochastically optimal. \\

\noindent 
\underline{Case 1}: $\SM$ selects $C_{m-1} = 1$ on the event $A$. \\
In this case, $\hat{\SM}$ selects $C_{m-1} = n-1$ on the event $A$, and thereby guarantees a length $m$ increasing 
subsequence $(\ell+1, \ell+2, ..., \ell + m = n)$ in the final permutation $\sigma$. Whereas, under $\SM$ one is only 
guaranteed a length $m-1$ longest increasing subsequence on the event $A$. Since there is positive probability that $A$ 
will occur, and $\SM$ and $\hat{\SM}$ behave identically if $A$ does not occur, it follows that
\begin{align*}
\P_{n,k}^{\hat{\SM}}(L \geq m) > \P_{n,k}^{\SM}(L \geq m).
\end{align*} 

\noindent 
\underline{Case 2}: $\SM$ selects $C_{m-1} = n-1$ on the event $A$. \\
In this case, $\hat{\SM}$ selects $C_{m-1} = 1$ on the event $A$, and there is some chance of ending with a 
length $\ell + 2$ increasing subsequence $(1,...,\ell,n-1,n)$ in the final permutation $\sigma$. Whereas, under 
$\SM$ the maximum possible length of increasing subsequence is only $\ell + 1$ on the event $A$. Since there 
is positive probability that $A$ will occur, and $\SM$ and $\hat{\SM}$ behave identically if $A$ does not occur, it follows that
\begin{align*}
\P_{n,k}^{\hat{\SM}}(L \geq \ell + 2) > \P_{n,k}^{\SM}(L \geq \ell + 2).
\end{align*} 

\noindent
\underline{Claim 2}: For all $k \geq 2, n \geq 4$ there exists a strategy $\SM_{copy}$, which is strictly better than $\SM_{\min}$ for maximizing $L$. \\ \\
\underline{Pf}:
Define $\SM_{copy}$ to be the strategy that uses exactly the same selection rules as $\SM_{\min}$ for $1 \leq t \leq n - 3$, 
and also the same selection rules as $\SM_{min}$ for $t \geq n-2$ unless both of the following conditions hold:
\begin{enumerate}
\item $C_1 = 2, C_2 = 3, ... , C_{n-3} = n - 2$. 
\item $\{C_{n-2,1},...,C_{n-2,k}\} = \{1,n-1\}$. That is, one is given only copies of cards $1$ and $n-1$ to pick from
at time $n-2$, and at least one copy of each. 
\end{enumerate}
In this case, the strategy $\SM_{copy}$ selects $C_{n-2} = n-1$ (instead of $C_{n-2} = 1$, as selected by $\SM_{\min}$). 

In the critical case when the two strategies select differently for $C_{n-2}$, $\SM_{copy}$ ensures a length $n-1$ 
increasing subsequence $(2,3,...,n)$ in the final permutation $\sigma$, whereas with $\SM_{\min}$ one is
only guaranteed a length $n-2$ longest increasing subsequence. Since no length $n$ increasing subsequence 
is possible in this instance, as $C_1 \not=1$, $\SM_{copy}$ is strictly better for maximizing $L$ in this instance than 
$\SM_{\min}$. In all other instances $\SM_{copy}$ and $\SM_{\min}$ behave identically, so the claim follows. 
\end{proof} 

\section*{Acknowledgments}
The author thanks Ross Pinsky, Ron Peled, and Jesse Goodman for helpful discussions and Ross Pinsky for suggestion of the model. 

\bibliography{ref}

\end{document}